\documentclass[11pt]{amsart}

\usepackage[paperheight=11in,
             paperwidth=8.5in,
             outer=1in,
             inner=1in,
             bottom=1in,
             top=1in]{geometry}
             
\usepackage{amsmath}
\usepackage{amsxtra}
\usepackage{amscd}
\usepackage{amsthm}
\usepackage{amsfonts}
\usepackage{amssymb}
\usepackage{eucal}
\usepackage[all]{xy}
\usepackage{graphicx}
\usepackage{pifont}
\usepackage{comment}
\usepackage{url}
\usepackage{enumitem}
\usepackage{mathrsfs}
\usepackage{color}
\usepackage{bbm}
\usepackage{breqn}

\usepackage[hidelinks]{hyperref}
\hypersetup{
    colorlinks,
    citecolor=magenta,
    filecolor=magenta,
    linkcolor=blue,
    urlcolor=black
}

\newcommand{\nc}{\newcommand}
\newcommand{\rc}{\renewcommand}

\nc\on{\operatorname}
\nc\ol{\overline}
\nc{\ul}{\underline}
\nc{\wt}{\widetilde}
\nc{\us}{\underset}
\nc{\os}{\overset}

\nc{\Spec}{\on{Spec}}
\nc\Id{\on{Id}}
\nc{\Hom}{\on{Hom}}
\nc{\res}{\on{res}}
\nc{\Cone}{\on{Cone}}
\nc{\Conv}{\on{Conv}}

\nc{\tcap}{\bar\pitchfork}
\nc{\del}{\nabla}

\renewcommand{\l}{\left}
\renewcommand{\r}{\right}
\nc{\la}{\langle}
\nc{\ra}{\rangle}
\nc{\rt}{\sqrt}
\nc{\injto}{\hookrightarrow}
\nc{\surjto}{\twoheadrightarrow}
\nc\mapsfrom{\mathrel{\reflectbox{\ensuremath{\mapsto}}}}
\nc\from{\mathrel{\reflectbox{\ensuremath{\to}}}}

\nc{\bs}{\bigskip}
\nc{\ms}{\medskip}

\nc{\noi}{\noindent}
\nc{\tn}{\textnormal}
\nc{\tb}{\textbf}
\nc{\mb}{\mathbf}
\nc{\dsp}{\displaystyle}
\nc{\tc}{\textcolor}

\rc{\AA}{\mathbb{A}}
\nc{\CC}{\mathbb{C}}
\nc{\FF}{\mathbb{F}}
\nc{\NN}{\mathbb{N}}
\nc{\PP}{\mathbb{P}}
\nc{\QQ}{\mathbb{Q}}
\nc{\RR}{\mathbb{R}}
\nc{\ZZ}{\mathbb{Z}}
\nc{\kk}{\mathbbm{k}}
\nc{\CL}{\mathcal{L}}
\nc{\CA}{\mathcal{A}}
\nc{\SA}{\mathscr{A}}
\nc{\SC}{\mathscr{C}}
\nc{\SF}{\mathscr{F}}
\nc{\SO}{\mathscr{O}}
\nc{\SG}{\mathscr{G}}
\nc{\SB}{\mathscr{B}}
\nc{\SN}{\mathscr{N}}
\nc{\SI}{\mathscr{I}}
\nc{\SL}{\mathscr{L}}

\nc{\al}{\alpha}
\nc{\be}{\beta}
\nc{\ga}{\gamma}
\nc{\lam}{\lambda}
\nc{\ep}{\epsilon}
\nc{\de}{\delta}
\nc{\De}{\Delta}
\nc{\ka}{\kappa}
\nc{\si}{\sigma}
\rc{\th}{\theta}
\nc{\vphi}{\varphi}
\nc{\pp}{\mathfrak{p}}
\nc{\mm}{\mathfrak{m}}
\nc{\vb}[1]{\vec{\bold{#1}}}

\theoremstyle{definition}
\newtheorem{cor}[subsection]{Corollary}

\newtheorem{prop}[subsection]{Proposition}
\newtheorem{thm}[subsection]{Theorem}
\newtheorem{defn}[subsection]{Definition}
\newtheorem{eg}[subsection]{Example}
\newtheorem{rem}[subsection]{Remark}

\newtheorem{question}{Question}

\title{Complete Intersections with given Hilbert Polynomials}
\author{Christopher Eur, Lim Sung Hyun}

\begin{document}

\maketitle

\begin{abstract}
The Hilbert polynomial of a homogeneous complete intersection is determined by the degrees of the generators of the defining ideal.  The degrees of the generators are not, in general, determined by the Hilbert polynomial -- but sometimes they are. When?  We give some general criteria and completely answer the question up to codimension 6.
\end{abstract}

\section{Introduction}

Let $X = X(a_1, \ldots, a_c)\subset \PP^n_\CC$ be a complete intersection of codimension $c$ in $\PP^n_\CC$ defined by homogeneous polynomials $(f_1, \ldots, f_c)$ of degrees $(a_1, \ldots, a_c)= (\deg f_1, \ldots, \deg f_c)$.  Can one recover up to symmetry the degree sequence $(a_1, \ldots, a_c)$ given only the Hilbert polynomial $P_X(t)$ of $X$?

\ms
The Hilbert polynomial $P_X(t) = \mu_0 t^{n-c} + \mu_1t^{n-c-1} + \cdots + \mu_{n-c}$ has $n-c+1$ coefficients,  so if the codimension $c$ is large relative to the ambient dimension $n$,  then one expects the answer to be no.  For example, consider codimension-2 complete intersections in the plane; $X(2,2)$, $X(1,4) \subset \PP^2$ both have Hilbert polynomial $P_X(t) = 4$.  However, for codimension-2 complete intersections in $\PP^3$, the degree sequence is always recoverable:  For $X = X(a_1, a_2) \subset \PP^3$ a complete intersection, one can recover $(a_1, a_2)$ from the degree $\deg X = a_1a_2$ and the arithmetic genus $p_a = \frac{1}{2}a_1a_2(a_1+a_2-4)+1$ and hence from the Hilbert polynomial $P_X(t) = (\deg X)t + (1-p_a)$.

\ms
We explore this behavior in general; for what values of codimension $c$ and ambient dimension $n$ do the Hilbert polynomial always determine the degree sequence of the complete intersection?

\ms
We say that a pair $(c,n)$ is \tb{firm} if the degree sequence of any codimension $c$ complete intersection in $\PP^n$ is determined by its Hilbert polynomial.  We will see in Proposition \ref{nplus1} that $(c,n)$ firm implies $(c, n+1)$ firm, so we are interested in 
$$N_c := \min\{n : (c,n) \tn{ firm}\} \quad(N_c:=\infty \tn{ if no such $n$ exists for $c$}).$$
Determining $N_c$ in general seems difficult.  In fact, we do not know whether $N_c$ is finite for all $c$.

\begin{question} Is $N_c$ finite for all $c$? \end{question}

Since $X(a_1, \ldots, a_c)\subset \PP^n$ has the same Hilbert polynomial as $X(1, a_1, \ldots, a_c)\subset \PP^{n+1}$, once we have $N_c = \infty$ for some $c$ we have $N_{c'}= \infty$ for any $c'\geq c$ also.  The difficulty in Question 1 stems from the degrees being unbounded; if the degrees in the degree sequence are bounded, then for sufficiently large $n$ the Hilbert polynomial determines the degree sequence:

\newtheorem*{cor:regdim}{Corollary \ref{regdim}}
\begin{cor:regdim} For a complete intersection $X(a_1, \ldots, a_c)\subset \PP^n$ such that $\on{reg} X \leq \dim X$ (equivalently $\sum_{i=1}^c a_i \leq n$), the Hilbert polynomial recovers the degree sequence.
\end{cor:regdim}

Our main result concerns $N_c$ for $c\leq 6$, proven in section \S4:

\newtheorem*{thm:Nc}{Theorem \ref{Nc}}
\begin{thm:Nc}
We have $N_1 = 1$, $N_2 = 3$, $N_3 = 5$, $N_4 = 8$, $N_5 = 11$, $N_6 = 14$.
\end{thm:Nc}

Moreover, we have the following parity restriction on $N_c$:

\newtheorem*{cor:evenodd2}{Corollary \ref{evenodd2}}
\begin{cor:evenodd2}
$N_c \equiv c \mod 2$ for $c>2$ and $N_c$ finite. 
\end{cor:evenodd2}

\noi which follows from a rigidity among the coefficients of Hilbert polynomials of a certain class of projective varieties which complete intersections are part of: 

\newtheorem*{thm:evenodd}{Theorem \ref{evenodd}}
\begin{thm:evenodd}
Let $X\subset \PP^n$ be a $d$-dimensional smooth projective variety whose total Chern class of the normal bundle $c(\mathscr N_{X/\PP^n})$ is a polynomial in $c_1(\SO_X(1))$ with $\ZZ$-coefficients.  Let $P_X(t) = \mu_0t^{d} + \mu_1t^{d-1} + \cdots \mu_{d-1}t + \mu_{d}$ be its Hilbert polynomial.  Then $\{\mu_{2i}\}_{2i\leq d}\cup \{\mu_1\}$ determine $P_X(t)$.
\end{thm:evenodd}

The other side of asking whether $N_c$ is finite is giving lower bounds for $N_c$.  Clearly for $c>1$ the condition $n>c$ is necessary for firmness.  Also, $P_X(t)$ has $n-c+1$ coefficients, but for $n>c$ Theorem \ref{evenodd} implies that  $P_X(t)$ has $\lceil \frac{n-c+1}{2}\rceil + 1$ ``irredundant" coefficients.  Thus, one may guess that $\lceil \frac{n-c+1}{2}\rceil + 1\geq c$ is necessary for $(c,n)$ to be firm, but since firmness is a matter of existence of positive solutions to a system of Diophantine equations, this is not clear.

For example, the two sequences $(46, 36, 32, 15, 12, 5), (45, 40, 24, 23, 8, 6)$ in Example \ref{c6eg} give the smallest example for non-firmness of (6,13) in the sense that the sum of the degrees is the minimum possible.  Note that there are ${51 \choose 6} =  18009460$ sequences of positive integers of length 6 (up to symmetry) with each entry $\leq 46$.

\begin{question} Suppose $c>1$ and $n>c$.  If $(c,n)$ is firm, then is $\dsp \left\lceil \frac{n-c+1}{2}\right\rceil + 1\geq c$? \end{question}

Theorem \ref{Nc} shows that $\left\lceil \frac{n-c+1}{2}\right\rceil + 1\geq c$ is necessary and sufficient condition for $(c,n)$ to be firm for $1<c\leq 6$ (given $n>c$).

\subsection{Structure of the paper}
In section \S2 we note that in the case of bounded degrees, $(c,n)$ is firm for large enough $n$.  In section \S3 we define quantities $\Lambda_i^c$'s which encode the same information as the Hilbert polynomial for complete intersections; these quantities have the advantage that they are independent of the ambient dimension $n$ and are simpler to compute.  In section \S4 we use these $\Lambda_i^c$'s to determine $N_c$ for $c\leq 6$.

\proof[Note on the ground field]  Throughout this paper we may assume the ground field to be any algebraically closed field of characteristic zero.  The Hilbert polynomial of a complete intersection is entirely determined the degree sequence regardless of the ground field, and for any field $k$ and degrees $(a_1, \ldots, a_c)$ there is at least one complete intersection in $\PP^n_k$ with the prescribed degrees, namely one given by the ideal $\langle x_0^{a_1}, x_1^{a_2}, \ldots, x_{c-1}^{a_c} \rangle \subset k[x_0, \ldots, x_n]$.

\section{The case of bounded degrees}

We start by noting that the Hilbert function, as opposed to the Hilbert polynomial, always recovers the degree sequence.

\begin{prop} Let $X(a_1, \ldots, a_c)\subset \PP^n$ be a complete intersection.  The Hilbert function of $X$ determines the degree sequence.
\end{prop}

\begin{proof}
Hilbert function determines the Hilbert series $H_X(t)$ of $X$, which for complete intersection $X(a_1, \ldots, a_c)$ is
$$H_X(t) = \frac{(1-t^{a_1})(1-t^{a_2})\cdots(1-t^{a_c})}{(1-t)^{n+1}}.$$
Without loss of generality, assume $a_1 \leq a_2 \leq \cdots\leq a_c$.  From the numerator of the Hilbert series of $X$ written with denominator $(1-t)^{n+1}$, the smallest nonzero degree of $t$ that appears  is $a_1$.  Divide the numerator then by $(1-t^{a_1})$, and continue the process.
\end{proof}

\begin{cor} \label{regdim} For complete intersection $X(a_1, \ldots, a_c)\subset \PP^n$ such that $\on{reg} X \leq \dim X$ (equivalently $\sum_i a_i \leq n$), the Hilbert polynomial recovers the degree sequence.
\end{cor}

\begin{proof}
Complete intersections are arithmetically Cohen-Macaulay, so the projective dimension of the coordinate ring $S_X$ of $X$ over $S = k[x_0, \ldots, x_n]$ is equal to the codimension $c$.  If $\on{reg} X \leq \dim X $, then \cite[Theorem 4.2.2]{Eis05} implies the Hilbert polynomial $P_X(t)$ is equal to the Hilbert function for $t \geq 0 \geq \on{reg} X +c -n = \on{reg} X - \dim X$.
\end{proof}

Thus, if the degrees in the degree sequence are bounded, then for a large enough ambient dimension $n$, the Hilbert polynomial always recovers the degree sequence.

\section{Modified Todd classes}

In this section we use the Hirzebruch-Riemann-Roch theorem to encode information equivalent to the Hilbert polynomial in a way that is both computationally and theoretically useful.

\subsection{The invariants $\Lambda^X_i$}  We introduce a certain collection of numbers which we call $\Lambda_i^X$'s that contain the same information as the Hilbert polynomial of $X$ for certain kinds of projective schemes $X$.  We first set up some notations regarding the Hirzebruch-Riemann-Roch theorem.

\medskip
Let $X\subset \PP^n_k$ be a $d$-dimensional smooth projective variety over an algebraically closed field $k$.  Denote by $A(X)$ its Chow ring graded by codimension and by $c(\mathscr E)\in A(X)$ the total Chern class of a coherent sheaf $\mathscr E$ on $X$.  For $\al \in A(X)$ denote by $\al_\ell$ its $\ell$-th graded part of $\al$.  Note that the embedding $i: X \injto \PP^n$ induces $i^*: A(\PP^n) \to A(X)$ where the hyperplane class $h := c_1(\SO_{\PP^n}(1)) \in A(\PP^n)$ pulls back to $h_X := c_1(\SO_X(1))$.  
Denote by $\int_X : A(X)^{d} \to \ZZ$ the degree map sending a class of a point $[\tn{pt}]\in A^{d}(X)$ to 1, and by $\deg X := \int_X h_X^d$ the degree of $X$ in $\PP^n$.  Lastly, by $\int_X \al$ we mean $\int_X \al_d$ where $\al_d$ is the $d$-th graded part of $\al \in A(X)$.

Recall that the Chern character $\on{Ch}$ and the Todd class $\on{Td}$ are two maps from $K(X)$, the Grothendieck group of vector bundles on $X$, to $A(X)$ given by the characteristic power series $Q(x) = e^x$ and $Q(x) = \frac{x}{1-e^{-x}}$ (respectively).  Here we will view them as group homomorphisms $(1+A^{>0}(X))\to (1+A^{>0}(X))$  where $1+A^{>0}(X)$ denotes multiplicative subgroup of $A(X)^*$ consisting of elements in $A(X)$ of the form $1+\al$ with $\al_0 = 0\in A^0(X)$; this is justified by the splitting principle.

\medskip
We recall the statement of the Hirzebruch-Riemann-Roch theorem.  For a topological proof, see \cite{Hir95}; a more algebraic proof of its generalization due to Grothendieck is given in \cite{BS58}.  For an exposition of its use and examples see \cite[Chapter 14]{EH16} or \cite[Appendix A]{Har77}.

\begin{thm}[HRR] \label{HRR} Let $\mathscr E$ be a locally free sheaf of finite rank on a smooth projective variety $X$ over an algebraically closed field $k$ with tangent bundle $\mathscr T_X$.  Then
$$\chi(\mathscr E) = \int_X \on{Ch}(\mathscr E) \cdot \on{Td}(\mathscr T_X),$$
and in particular, if $X\subset \PP^n$ then its Hilbert polynomial $P_X(t)$ is
$$P_X(t) = \chi(\SO_X(t)) = \int_X \on{Ch}(\SO_X(t)) \cdot \on{Td}(\mathscr T_X).$$

\end{thm}

We now modify the Hirzebruch-Riemann-Roch theorem slightly to encode information equivalent to the Hilbert polynomial in a way that is convenient for analyzing complete intersections.  We first restrict to a subclass of projective varieties satisfying the following condition:

\theoremstyle{remark}
\newtheorem{cond}{Condition}
\renewcommand{\thecond}{$(*)$}
\begin{cond}\label{star}
$X$ is a $d$-dimensional smooth projective variety over an algebraically closed field $k$ whose total Chern class of the normal bundle $c(\mathscr N_{X/\PP^n})$ satisfies $c(\mathscr N_{X/\PP^n}) \in \ZZ[h_X] = i^*(A(\PP^n))$.
\end{cond}

Note that a complete intersection $X(a_1, \ldots, a_c)\subset \PP^n_k$ certainly satisfy \ref{star} since $c(\mathscr N_{X/\PP^n}) = \prod_{i=1}^c (1+a_ih_X)$ (\cite[Example 5.19]{EH16}).

\begin{defn}
Let $X\subset \PP^n$ satisfy the condition \ref{star}.  Define
$$\Lambda_i^X := \int_X h_X^{d-i}\cdot \on{Td}\l(\frac{1}{c(\mathscr N_{X/\PP^n})}\r)_i$$
where $\mathscr N_{X/\PP^n}$ is the normal bundle of $X$ in $\PP^n$.
\end{defn}

We'll drop the superscript $X$ when it is understood.  First, we note that knowing $\Lambda_i$'s  is the same as knowing the Hilbert polynomial for $X$ satisfying $(*)$.

\begin{prop}\label{lambda}
Let $X\subset \PP^n$ satisfy \ref{star} and let $P_X(t) = \mu_0t^d + \cdots +\mu_{d-1}t + \mu_d$ be its Hilbert polynomial.  Then $(\Lambda_i)_{0\leq i \leq d}$ and $(\mu_i)_{0\leq i \leq d}$ are related by a lower triangular matrix whose entries depend only on $d$, $n$, and $\deg X$.
\end{prop}

\begin{proof}
As we have a short exact sequence $0 \to \mathscr T_X \to \mathscr T_{\PP^n}|_X \to \mathscr N_{X/\PP^n}\to 0$, we have $c(\mathscr T_X) = \frac{c(\mathscr T_{\PP^n})}{c(\mathscr N_{X/\PP^n})} = \frac{(1+h_X)^{n+1}}{c(\mathscr N_{X/\PP^n})}$.  Then by the Hirzebruch-Riemann-Roch theorem \ref{HRR} we have

$$\begin{array}{rl}
P_X(t) &= \int_X \on{Ch}(\SO_X(t)) \cdot \on{Td}(\mathscr T_X) \\[3mm]
&= \int_X (1+h_Xt + \frac{h_X^2t^2}{2!} + \cdots + \frac{h_X^dt^d}{d!})\cdot \on{Td}(\frac{1}{c(\mathscr N_{X/\PP^n})}) \cdot \on{Td}((1+h_X)^{n+1}).\\[3mm]
\end{array}
$$
Note that $\on{Td}((1+h_X)^{n+1})$ is a polynomial in $h_X$ that depend only on $d$ and $n$.  Thus, denoting $\operatorname{Td}(1+h_X)^{n+1}) = 1 + \frac{(n+1)}{2}h_X + \cdots = \sum_{i = 0}^d q_i h_X^i$, we have that
$$\begin{bmatrix}
\frac{1}{d!} & 0 & 0 & 0 & \cdots & 0\\[1mm]
\frac{q_1}{(d-1)!} & \frac{1}{(d-1)!} & 0 & 0 & \cdots & 0\\
\frac{q_2}{(d-2)!} & \frac{q_1}{(d-2)!} & \frac{1}{(d-2)!} & 0 & \cdots & 0\\
\vdots & \vdots & \vdots & \ddots &  & \vdots\\
\vdots & \vdots & \vdots &  & \ddots & \vdots\\
q_d & q_{d-1} & q_{d-2} & \cdots & \cdots & 1
\end{bmatrix}
\begin{bmatrix}
\Lambda_0\\
\Lambda_1\\
\Lambda_2\\
\Lambda_3\\
\vdots\\
\vdots\\
\Lambda_d
\end{bmatrix}
=
\begin{bmatrix}
\mu_0\\
\mu_1\\
\mu_2\\
\mu_3\\
\vdots\\
\vdots\\
\mu_d
\end{bmatrix}
$$
\end{proof}

\ms Kleiman had shown that the coefficients of the Hilbert polynomial of a projective variety satisfy certain rigid structure; for details see \cite{Kle71}.  For $X$ satisfying \ref{star},  we note another rigidity in the coefficients of the Hilbert polynomial.

\begin{thm}\label{evenodd}
Let $X\subset \PP^n$ be a $d$-dimensional smooth projective variety satisfying \ref{star}.  Let $P_X(t) = \mu_0t^{d} + \mu_1t^{d-1} + \cdots \mu_{d-1}t + \mu_{d}$ be its Hilbert polynomial.  Then for each $\ell  \leq d$ odd, $\mu_\ell$ is a $\QQ$-linear combination (independent of $X$) of $\l\{\frac{\mu_1^{\ell-i}\mu_{i}}{\mu_0^{\ell-i}}\r\}_{i<\ell }$.  In particular, for each $\ell\leq d$ odd, $\{\mu_{2i}\}_{2i<\ell}$ together with $\mu_1$ determine $\mu_\ell$; likewise, $\{\Lambda_{2i}\}_{2i<\ell}$ together with $\Lambda_1$ determine $\Lambda_\ell$.
\end{thm}

\begin{proof}
Since we have $c(\mathscr T_X) = \frac{c(\mathscr T_{\PP^n})}{c(\mathscr N_{X/\PP^n})} = \frac{(1+h_X)^{n+1}}{c(\mathscr N_{X/\PP^n})}$, the condition \ref{star} implies that the canonical bundle $\omega_X = \det \mathscr T_X^\vee$ of $X$ satisfies $c(\omega_X) = 1- c_1(\mathscr T_X) = c(\SO_X(k))$ for some $k\in \ZZ$.   Then Serre duality gives $h^i(X, \SO_X(t)) = h^{d-i}(X, \omega_X \otimes \SO_X(-t))$ for $t\in \ZZ$, so that we have
$$P_X(t) = \chi(\SO_X(t)) = (-1)^d \chi(\SO_X(k-t)) = (-1)^dP_X(k-t)$$
so that for each $0\leq \ell \leq d$ we have
$$P_X^{(d-\ell)}(t) = (-1)^\ell P_X^{(d-\ell)}(k-t)$$
where $P_X^{(m)}(t)$ is the $m$-th formal derivative of $P_X(t)$.  Hence,
\begin{equation}
(d-\ell)!\mu_\ell = P_X^{(d-\ell)}(0) = (-1)^{\ell}P_X^{(d-\ell)}(k).
\end{equation}
Now, as a degree $\ell$ polynomial $f(t)$ satisfies $f(t) = f^{(\ell)}(0)\frac{t^\ell}{\ell!} + f^{(\ell - 1)}(0) \frac{t^{\ell-1}}{(\ell-1)!} + \cdots  +f^{(1)}(0) t + f(0)$, we have a manipulation
\begin{equation}
\begin{array}{rl}
P_X^{(d-\ell)}(0) &= \Big( P_X^{(d-\ell)}(0) - P_X^{(d-\ell)}(k) \Big) + P_X^{(d-\ell)}(k)\\[2mm]
&=\dsp  -\Big(P_X^{(d)}(0)\frac{k^\ell}{\ell!} + P_X^{(d-1)}(0)\frac{k^{\ell-1}}{(\ell-1)!}+ \cdots + P_X^{(d-\ell+1)}(0)k\Big) + P_X^{(d-\ell)}(k).
\end{array}
\end{equation}
For $\ell$ odd, adding (1) with (2) gives
\begin{equation}
\begin{array}{rl}
2(d-\ell)!\mu_\ell = 2P_X^{(d-\ell)}(0) &=  -\Big(P_X^{(d)}(0)\frac{k^\ell}{\ell!} + P_X^{(d-1)}(0)\frac{k^{\ell-1}}{(\ell-1)!}+ \cdots + P_X^{(d-\ell+1)}(0)k\Big)\\[3mm]
& = -\Big(\frac{d!}{\ell!}k^\ell\mu_0 + \frac{(d-1)!}{(\ell-1)!}k^{\ell-1}\mu_1+ \cdots + (d-\ell+1)!k\mu_{\ell-1}\Big).
\end{array}
\end{equation}
Lastly, by Hirzebruch-Riemann-Roch Theorem \ref{HRR}, we have
$$\mu_1= \int_X \frac{h_X^{d-1}}{(d-1)!} \on{Td}(\mathscr T_X)_1= \int_X \frac{h_X^{d-1}}{(d-1)!} \frac{c_1(\mathscr T_X)}{2} = \int_X \frac{h_X^{d-1}}{(d-1)!} \cdot\frac{-kh_X}{2} = -\frac{(\deg X)k}{2(d-1)!} = \frac{-d}{2} \mu_0k.$$

Combining this with (3), we have that $\mu_\ell$ for $\ell$ odd is a $\QQ$-linear combination of  $\{\frac{\mu_1^{\ell-i}\mu_{i}}{\mu_0^{\ell-i}}\}_{i < \ell}$ where the coefficients only depend on $d$ and $\ell$.  The second statement of the theorem about $\Lambda_i$'s then follows immediately from the first via Proposition \ref{lambda}.
\end{proof}

\subsection{$\Lambda_i^X$'s for $X$ a complete intersection}  We now discuss the use of $\Lambda_i^X$'s in analyzing the complete intersection cases.

\ms
Let $X:= X(a_1, \ldots, a_c)\subset \PP^n$ be a complete intersection.  In regards to the firmness of $(c,n)$, it is harmless to assume that $X$ is smooth via the classical Bertini theorem as we assume our ground field to be algebraically closed with characteristic zero.  Then, as $c(\mathscr N_{X/\PP^n}) = \prod_{i=1}^c(1+a_ih_X)$, we have that $X$ satisfies the condition \ref{star}.    In the case of complete intersections, we have $\frac{1}{c(\mathscr N_{X/\PP^n})} = \prod_{i=1}^c \frac{1}{(1+a_ih_X)} = \sum_{j = 0}^{n-c} C_j(\ul a) h_X^j$ where $C_j(\ul a)$ is $(-1)^j$ times the degree $j$ complete homogeneous symmetric polynomial in $(a_1, \ldots, a_c)$.

\ms
We now recall the $j$-th formal Todd polynomial $T_j(\epsilon_1, \ldots,\epsilon_j)$ in formal variables $\epsilon_1, \ldots, \epsilon_j$, which is obtained by the following process.  Let $b_1, \ldots, b_j$ be formal variables and set $\epsilon_1, \ldots, \epsilon_j$ be their elementary symmetric functions, and consider the power series $\wt T:= \prod_{i=1}^j \frac{b_i}{1-e^{-b_i}}\in \QQ[[b_1, \ldots, b_j]]$; then, define $T_j$ to be the degree $j$ part of $\wt T$ written in terms of the elementary symmetric functions $\epsilon_1, \ldots, \epsilon_j$.  For example, $T_2 = \frac{\epsilon_1^2 + \epsilon_2}{12}$ since $(\frac{b_1}{1-e^{-b_1}})(\frac{b_2}{1-e^{-b_2}}) = (1+ \frac{b_1}{2} + \frac{b_1^2}{12} + \cdots)(1+ \frac{b_2}{2} + \frac{b_2^2}{12} + \cdots)$ and $b_1^2 + 3b_1b_2 + b_2^2 = (b_1+b_2)^2 + b_1b_2$.  For $\al = \sum_{i=0}^d \al_i h_X \in \QQ[h_X] \subset A(X)$ with $\al_0 = 1$, we have $T_i(\al_1, \ldots, \al_i)\in \QQ$ given by $\on{Td}(\al)_i = T_i(\al_1, \ldots, a_i) h_X^i$.

\medskip
Now, noting the value of $\frac{1}{c(\mathscr N_{X/\PP^n})}$ for complete intersections as discussed above, we have:

\begin{prop}\label{nindep}
Let $X(a_1, \ldots, a_c)\subset \PP^n$ be a (smooth) complete intersection.  Then 
$$\Lambda_i^X = (\deg X)T_i(C_1(\ul a), \ldots, C_i(\ul a))$$
where $C_j(\ul a)$ is the degree $j$ complete homogeneous symmetric polynomial in $(a_1, \ldots, a_c)$.    In particular, the values of $\Lambda_i^X$ for $0\leq i \leq n-c$ are independent of the ambient dimension $n$.
\end{prop}

As an immediate corollary, combining Proposition \ref{nindep} above with Proposition \ref{lambda} gives us

\begin{prop}\label{nplus1}
If $(c,n)$ is firm, then so is $(c, n+1)$.
\end{prop}

\begin{cor}\label{evenodd2}
Suppose $c>2$, and $c\equiv n \mod 2$.  Then $(c,n)$ not firm implies $(c, n+1)$ not firm.  $N_c \equiv c \mod 2$ for $c>2$ and $N_c$ finite.
\end{cor}

\begin{proof}
As $n-c$ is even, this implies that $\{\Lambda_i\}_{i\leq n-c}$ determine $\Lambda_{n+1-c}$ by Theorem \ref{evenodd}.
\end{proof}

\ms
Lastly, in preparation for the next section, we define $\wt \Lambda_i^c$ as follows.
\begin{defn} Fix $c\in \NN$ and let  $\underline \xi = (\xi_1, \ldots, \xi_c)$ be formal variables, and set $e_j$ for $j = 1, \ldots, c$ to be the $j$-th elementary symmetric polynomial in the $\xi_i$'s.  Define for $0\leq i \leq n-c$ the polynomial $\wt \Lambda_i^c\in \QQ[e_1, \ldots, e_c]$ by
$$
\wt\Lambda_i^c := \left\{
\begin{array}{lr}
e_c & \tn{if $i = 0$}\\
T_i(C_1(\ul \xi), \ldots, C_i(\ul \xi))\tn{ written in elementary symmetric polynomials} & \tn{if $i>0$}
\end{array}\right.
$$
\end{defn}

Proposition \ref{nindep} implies the following two facts about $\widetilde\Lambda_i^c$.  When $(\xi_1, \ldots, \xi_c) = (a_1, \ldots, a_c)$ for a complete intersection $X = X(a_1, \ldots, a_c)\subset \PP^n$, the coefficients $\{\mu_i\}_{0\leq i \leq n-c}$ of $P_X(t)$ is again related to $\{\widetilde\Lambda_i^c\}_{0\leq i \leq n-c}$ by a lower triangular matrix.  Moreover, $\wt \Lambda_i^c$ does not depend on $n$.

\medskip
Sometimes, we will convert from writing symmetric functions in elementary symmetric polynomials to writing them in \emph{monomial symmetric polynomials} $M_\lambda$ as defined in \cite[Appendix A]{FH91}:  for $\lambda = (\lambda_1, \ldots, \lambda_p)$ a partition of $d$ with $p\leq c$ parts, define $M_\lambda\in \QQ[a_1, \ldots, a_c]$ as $M_\lambda := \sum a_{i_1}^{\lambda_1}\cdots a_{i_p}^{\lambda_p}$.  For the conversions between the two families of symmetric polynomials, see \cite[I.6]{Mac15}.

\begin{rem}
As with symmetric functions, there are other ways to encode the information in the Hilbert polynomial $P_X(t)$ of a complete intersection such that similar properties in Proposition \ref{lambda} and Proposition \ref{nindep} are satisfied, but they seem to make for a lengthier exposition and computation.  For example, as the Koszul complex of the defining homogeneous forms of a complete intersection forms the minimal free resolution of the coordinate ring, one can interpret the resulting Hilbert polynomial and successive first differences as some combinatorial quantities via inclusion-exclusion principle.  Another way is to write the Hilbert polynomial as sums of Hilbert polynomials of projective spaces $P_X(t) = \sum_i c_i P_{\PP^i}(t)$.
\end{rem}

\section{Codimension $\leq$6 cases}

In this section, we prove the following theorem.

\begin{thm}\label{Nc}
We have $N_1 = 1$, $N_2 = 3$, $N_3 = 5$, $N_4 =
8$, $N_5 = 11$, $N_6 = 14$.
\end{thm}

\begin{rem}
The computations for $\wt \Lambda_i$'s and other relevant equations in this sections were done on Macaulay2.  The code used is available on the first author's website at \url{https://math.berkeley.edu/~ceur/code/ciHilbPoly.m2}. The command most used is \verb+ciToddPoly(c,n)+, which 
prints out $\wt\Lambda_i^c$ for $0\leq i \leq n-c$.  

To find an example showing that (4,7), (5,10), and (6,13) are not firm, a naive search by comparing Hilbert polynomials for various sequences proves unfruitful.  For example, two sequences given in Example  \ref{c6eg} are the smallest in the sense that the sum of the entries in degree sequences is minimum possible, but there are ${51 \choose 6} =  18009460$ sequences of positive integers of length 6 (up to symmetry) with each entry $\leq 46$.  A more efficient search algorithm exploiting the nature of $\wt\Lambda_i$'s is implemented as \verb+iterSearch+.
\end{rem}

\subsection{Codimension $\leq 4$} That $N_1 = 1$ is obvious and $N_2 = 3$ was done in the introduction.  For $N_3$, we have

\begin{dmath*}
\wt\Lambda^3_0 = e_3
\end{dmath*}

\begin{dmath*} \wt\Lambda^3_1 = ({e}_{1}) (-1/2)
\end{dmath*}

\begin{dmath*} \widetilde\Lambda^3_2 = (2 {e}_{1}^{2}-{e}_{2}) (1/12)
\end{dmath*}
so that $e_1, e_2, e_3$ are recoverable.  Hence $(3,5)$ is firm, and $N_3 =5$ is the minimum since $(3,4)$ is not firm as the following example shows:

\begin{eg}\label{eg1} Consider two complete intersection curves $X(2,5,9), X(3,3,10) \subset \PP^4$.  Both have $P_X(t) = 90t - 495$ as the Hilbert polynomial.
\end{eg}

Up to $c \leq 3$, the condition $n-c+1 \geq c$ is necessary and sufficient condition for $(c,n)$ to be firm.  Starting with $c = 4$ however, $n-c+1 \geq c$ is no longer sufficient.

\begin{eg}
Consider $X(2,6,7,15)$ and $X(3,3,10,14)$ in $\PP^7$, which both have Hilbert polynomial $210 t^{3}-6930 t^{2}+92295 t-456225$.
\end{eg} 

In fact, Theorem \ref{evenodd} explains why.  Computing $\wt\Lambda^4_i$ for $i = 0,1,2,3$ we get

\begin{dmath*}
\widetilde\Lambda_0^4 = e_4
\end{dmath*}

\begin{dmath*}  \widetilde\Lambda_1^4 = ({e}_{1})  (-1/2)
\end{dmath*}

\begin{dmath*} \widetilde\Lambda_2^4 = (2 {e}_{1}^{2}-{e}_{2}) (1/12)
\end{dmath*}

\begin{dmath*} \widetilde\Lambda_3^4 = ({e}_{1})   ({e}_{1}^{2}-{e}_{2}) (-1/24)
\end{dmath*}
As expected from Theorem \ref{evenodd}, $\wt \Lambda_3^4$ is determined from $\wt \Lambda_0^4, \wt \Lambda_1^4, \wt \Lambda_2^4$.  While we can determine $e_4, e_1, e_2$ of $(a_1, \ldots, a_4)$, we see that $e_3$ is not recoverable from the Hilbert polynomial in this case.  The two sequences (2,6,7,15) and (3,3,10,14) have the same $e_1, e_2, e_4$ but different $e_3$.

Increasing $n$ to 8, one computes that $\wt\Lambda_4^4 = (6 {e}_{1}^{4}-9 {e}_{1}^{2} {e}_{2}+2{e}_{2}^{2}-{e}_{1} {e}_{3}+{e}_{4}) (1/720)$,  so that

\begin{prop}
$(4,8)$ is firm.
\end{prop}

\begin{proof}
We already know that $e_4, e_1, e_2$ are determined, and $\wt \Lambda_4^4$ is of degree 1 with respect to $e_3$.
\end{proof}

\subsection{Codimension $=$ 5 case} We first confirm by example that $(5,10)$ is not firm.

\begin{eg}
By Proposition \ref{evenodd2}, it suffices to show that $(5,9)$ is not firm.  Let $X(4,4,15,15,22)$ and $X(3,6,11,20,20)$ be two complete intersections in $\PP^9$.  They have Hilbert polynomial $3300 t^{4}-330000 t^{3}+13952400 t^{2}-285120000 t+2328530380$.  As expected, the same pair of sequences shows that $(5,10)$ is not firm; they both have Hilbert polynomial $660 t^{5}-80850 t^{4}+4486900 t^{3}-135666300 t^{2}+2188295670 t-14860251560$ as complete intersections in $\PP^{10}$.
\end{eg}

\begin{prop}
$(5,11)$ is firm.
\end{prop}

\begin{proof}
We have that $\wt \Lambda_0^5, \wt \Lambda_1^5, \wt \Lambda_2^5, \wt \Lambda_4^5, \wt \Lambda_6^5$ are
\begin{dmath*}
\wt \Lambda_0^5 = e_5
\end{dmath*}

\begin{dmath*} \wt \Lambda_1^5 = ({e}_{1}) (-1/2)
\end{dmath*}

\begin{dmath*} \wt \Lambda_2^5 = (2 {e}_{1}^{2}-{e}_{2}) (1/12)
\end{dmath*}

\begin{dmath*} \wt \Lambda_4^5 = (6 {e}_{1}^{4}-9 {e}_{1}^{2}
      {e}_{2}+2 {e}_{2}^{2}-{e}_{1} {e}_{3}+{e}_{4}) (1/720)
\end{dmath*}

\begin{dmath*} \wt \Lambda_6^5 = (12 {e}_{1}^{6}-30
      {e}_{1}^{4} {e}_{2}+24 {e}_{1}^{2} {e}_{2}^{2}-12 {e}_{1}^{3} {e}_{3}-3
      {e}_{2}^{3}+3 {e}_{1} {e}_{2} {e}_{3}+12 {e}_{1}^{2} {e}_{4}+{e}_{3}^{2}-5
      {e}_{2} {e}_{4}+2 {e}_{1} {e}_{5}) (1/60480)
\end{dmath*}

As usual, we know that $e_5, e_1, e_2$ can be recovered from $\wt \Lambda^5_0, \wt \Lambda^5_1, \wt \Lambda^5_2$.  We note that $\wt \Lambda_4^5$ is linear in $e_3, e_4$ as $-e_3e_1 + e_4  = 720 \wt \Lambda_4^5 - 6( {e}_{1}^{4}-9 {e}_{1}^{2} {e}_{2}+2 {e}_{2}^{2})$.  As the value of $K :=  720\wt \Lambda_4^5 - 6( {e}_{1}^{4}-9 {e}_{1}^{2} {e}_{2}+2 {e}_{2}^{2})$  is known, we substitute $e_4 = e_3e_1 +K$ into $\wt \Lambda_6^4$ to get a quadric equation in terms of $e_3$:
$$
{e}_{3}^{2}-2 {e}_{1} {e}_{2} {e}_{3} +12 {e}_{1}^{6}-30 {e}_{1}^{4} {e}_{2}+24 {e}_{1}^{2} {e}_{2}^{2}-3{e}_{2}^{3}+12 {e}_{1}^{2} K+2 {e}_{1} {e}_{5}-5 {e}_{2} K = 60480 \wt \Lambda_6^5
$$

This shows that there are potentially two values of $e_3$, but the average of the two values for $e_3$ is $e_1e_2$.  Since $e_1e_2-e_3 = M_{2,1}+3M_{1,1,1} - M_{1,1,1} = M_{2,1}+ 2M_{1,1,1}>0$ for any elementary symmetric functions of positive integers, we take the smaller root, and recover $e_3$.
\end{proof}

\subsection{Codimension $=$ 6 case} We first list $\wt \Lambda_i^6$ for $i = 0,1,2,4,6,8$.

\begin{dmath*}
\wt \Lambda_0^6  = e_6
\end{dmath*}

\begin{dmath*} \wt \Lambda_1^6  = ({e}_{1}) (-1/2)
\end{dmath*}

\begin{dmath*} \wt \Lambda_2^6  = (2 {e}_{1}^{2}-{e}_{2}) (1/12)
\end{dmath*}

\begin{dmath*} \wt \Lambda_4^6  = (6 {e}_{1}^{4}-9 {e}_{1}^{2}
      {e}_{2}+2 {e}_{2}^{2}-{e}_{1} {e}_{3}+{e}_{4}) (1/720)
\end{dmath*}

\begin{dmath*} \wt \Lambda_6^6  = (12 {e}_{1}^{6}-30
      {e}_{1}^{4} {e}_{2}+24 {e}_{1}^{2} {e}_{2}^{2}-12 {e}_{1}^{3} {e}_{3}-3
      {e}_{2}^{3}+3 {e}_{1} {e}_{2} {e}_{3}+12 {e}_{1}^{2} {e}_{4}+{e}_{3}^{2}-5
      {e}_{2} {e}_{4}+2 {e}_{1} {e}_{5}-2 {e}_{6}) (1/60480)
\end{dmath*}

\begin{dmath*} \wt \Lambda_8^6  = (10 {e}_{1}^{8}-35
      {e}_{1}^{6} {e}_{2}+50 {e}_{1}^{4} {e}_{2}^{2}-25 {e}_{1}^{5} {e}_{3}-25
      {e}_{1}^{2} {e}_{2}^{3}+25 {e}_{1}^{3} {e}_{2} {e}_{3}+25 {e}_{1}^{4} {e}_{4}+2
      {e}_{2}^{4}-3 {e}_{1} {e}_{2}^{2} {e}_{3}+9 {e}_{1}^{2} {e}_{3}^{2}-42
      {e}_{1}^{2} {e}_{2} {e}_{4}+17 {e}_{1}^{3} {e}_{5}-2 {e}_{2} {e}_{3}^{2}+7
      {e}_{2}^{2} {e}_{4}-{e}_{1} {e}_{3} {e}_{4}-4 {e}_{1} {e}_{2} {e}_{5}-17
      {e}_{1}^{2} {e}_{6}+2 {e}_{4}^{2}-3 {e}_{3} {e}_{5}+7 {e}_{2} {e}_{6})
      (1/3628800)
\end{dmath*}

\begin{eg}\label{c6eg}
Again by Proposition \ref{evenodd2}, we only need show that $(6,12)$ is not firm to show that $N_6 \geq 14$.  From the first five polynomials listed above, we see that this amounts to finding two different positive integer sequences of length 6 such that $e_6, e_1, e_2, e_1e_3 - e_4, e_3^2 - 2e_2e_4 + 2e_1e_5$ are the same.  A computer search via \verb+iterSearch(6,150,6)+ produces
$$X(46, 36, 32, 15, 12, 5), \quad X(45, 40, 24, 23, 8, 6)$$
which both have Hilbert polynomial:
\begin{dmath*} P_X(t) = 66240 t^{6}-26429760 t^{5}+4792795200 t^{4}-495690148800 t^{3}+30434011089120 t^{2}-1041907113767520 t+15429613604601120
\end{dmath*}
as complete intersections in $\PP^{12}$.  This is the ``smallest" example in the sense that the sum of the degree sequence ($=146$) is the minimum possible.
\end{eg}

\begin{prop}
$(6,14)$ is firm.
\end{prop}

\begin{proof}
As before, $e_6, e_1, e_2$ are recovered from $\wt \Lambda_0^6, \wt \Lambda_1^6, \wt \Lambda_2^6$ as usual.  Just as in the $c=5$ case, $\wt \Lambda_4^6$ is linear in $e_3, e_4$ so that we again make the substitution $e_4 = e_3e_1 +K$ into $\wt \Lambda_6^6$ and $\wt \Lambda_8^6$ to obtain

\begin{dmath*}
\wt \Lambda_6^6 = (12 {e}_{1}^{6}-30 {e}_{1}^{4} {e}_{2}+24 {e}_{1}^{2} {e}_{2}^{2}-3
      {e}_{2}^{3}-2 {e}_{1} {e}_{2} {e}_{3}+12 {e}_{1}^{2} K+{e}_{3}^{2}+2 {e}_{1}
      {e}_{5}-5 {e}_{2} K) (1/60480)
\end{dmath*}

\begin{dmath*}
\wt \Lambda_8^6 = (10 {e}_{1}^{8}-35 {e}_{1}^{6} {e}_{2}+50
      {e}_{1}^{4} {e}_{2}^{2}-25 {e}_{1}^{2} {e}_{2}^{3}-17 {e}_{1}^{3} {e}_{2}
      {e}_{3}+25 {e}_{1}^{4} K+2 {e}_{2}^{4}+4 {e}_{1} {e}_{2}^{2} {e}_{3}+10
      {e}_{1}^{2} {e}_{3}^{2}+17 {e}_{1}^{3} {e}_{5}-42 {e}_{1}^{2} {e}_{2} K-2
      {e}_{2} {e}_{3}^{2}-4 {e}_{1} {e}_{2} {e}_{5}+7 {e}_{2}^{2} K+3 {e}_{1} {e}_{3}
      K-3 {e}_{3} {e}_{5}+2 K^{2}) (1/3628800)\end{dmath*}

Noting that the first equation is linear in $e_5$, we can make another substitution $e_5 = -\frac{e_3^2}{2e_1} + e_2e_3 + K'$ into the second equation (where $K'$ is a known quantity).  We then obtain

\begin{dmath*}
7257600e_1\wt \Lambda_8^6 = 3 {e}_{3}^{3}+(3 {e}_{1}^{3}-6 {e}_{1} {e}_{2}) {e}_{3}^{2}+(6 {e}_{1}^{2} K-6
      {e}_{1} K') {e}_{3}+20 {e}_{1}^{9}-70 {e}_{1}^{7} {e}_{2}+100
      {e}_{1}^{5} {e}_{2}^{2}-50 {e}_{1}^{3} {e}_{2}^{3}+50 {e}_{1}^{5} K+4 {e}_{1}
      {e}_{2}^{4}-84 {e}_{1}^{3} {e}_{2} K+34 {e}_{1}^{4} K'+14 {e}_{1}
      {e}_{2}^{2} K-8 {e}_{1}^{2} {e}_{2} K'+4 {e}_{1} K^{2}
\end{dmath*}

Call the left-hand-side of above $f:= 7257600e_1\wt \Lambda_8^6$, and view the right-hand-side as a cubic form $F$ in $e_3$.  We wish to pick out which root of $F(e_3) = f$ is the true value of $e_3$.  We first note that $2e_1e_2 - e_1^3 = 2(M_{2,1}+ 3M_{1,1,1}) - M_3 - 3M_{2,1} - 6M_{1,1,1} = -M_3 - M_{2,1}<0$.  Thus, we have that the average of the three roots of $F=f$ is negative.  Hence, we are done once we know the sign of the larger root of $F'(t) = 0$ minus $e_3$.  Well, we have $\frac13 F'(t) = 3t^2 + 2(e_1^3 - 2e_1e_2)t + 2e_1^2K - 2e_1K'$.  Thus, the larger root of $F'(t) = 0$ minus $e_3$ is:
$$\frac{-3e_3 + (2e_1e_2 - e_1^3) + \rt{(2e_1e_2 - e_1^3)^2 - 3(2e_1^2K - 2e_1K')}}{3}$$
Since $-3e_3+ (2e_1e_2 - e_1^3)<0$, the sign of the above is determined by the sign of
$$-(-3e_3+ (2e_1e_2 - e_1^3))^2 +\big((2e_1e_2 - e_1^3)^2 - 3(2e_1^2K - 2e_1K')\big)$$
Using $K = -e_1e_3 + e_4$ and $K' = \frac{e_3^2}{2e_1} - e_2e_3 + e_5$, the above simplifies to the following (after division by 6):
$$-{e}_{3}^{2}+{e}_{1} {e}_{2} {e}_{3}-{e}_{1}^{2} {e}_{4}+{e}_{1} {e}_{5} = (e_1e_2 -e_3)e_3 - e_1(e_1e_4 - e_5).$$
We claim that the above quantity is always positive as a symmetric function of $(a_1, \ldots, a_6)$ when $a_i$'s are positive real numbers.  To see this, we convert the above expression into monomial symmetric polynomials:
$$\begin{array}{rl}
(e_1e_2 -e_3)e_3 - e_1(e_1e_4 - e_5) =& M_{1,1,1}(M_{2,1}+3M_{1,1,1} -M_{1,1,1})\\ &- M_1(M_{2,1,1,1}+5M_{1,1,1,1,1} - M_{1,1,1,1,1})\\[2mm]
=& 2M_{1,1,1}^2 + M_{2,1}M_{1,1,1} - M_1M_{2,1,1,1} - 4M_1M_{1,1,1,1,1}\\[2mm]
=&2{6 \choose 3}M_{1,1,1,1,1,1} + 2{4\choose 2}M_{2,1,1,1,1} + 2{2 \choose 1}M_{2,2,1,1} + 2M_{2,2,2}\\
& +4M_{2,1,1,1,1} + 2M_{2,2,1,1} + M_{3,1,1,1} + M_{3,2,1}\\
& -(4M_{2,1,1,1,1} + 2M_{2,2,1,1} + M_{3,1,1,1})\\
& - 4(6M_{1,1,1,1,1,1}+M_{2,1,1,1,1})\\[2mm]
=& 16M_{1,1,1,1,1,1}+8M_{2,1,1,1,1} + 4M_{2,2,1,1}+2M_{2,2,2}+M_{3,2,1}>0
\end{array}
$$
as we note that
$$M_{1,1,1}^2 = {6 \choose 3}M_{1,1,1,1,1,1} + {4\choose 2}M_{2,1,1,1,1} + {2 \choose 1}M_{2,2,1,1} + M_{2,2,2}$$
$$M_{2,1}M_{1,1,1} = 4M_{2,1,1,1,1} + 2M_{2,2,1,1} + M_{3,1,1,1} + M_{3,2,1}$$
$$M_1M_{2,1,1,1} = 4M_{2,1,1,1,1} + 2M_{2,2,1,1} + M_{3,1,1,1}$$

\end{proof}

\subsection{Codimension $ = 7$ case} We conclude with some remarks about determining $N_7$.  Taking the same approach taken in $c=6$ case, we compute $\widetilde\Lambda^7_i$ for $i = 0,1,2,4,6,8,10$ and make some substitutions.  One then obtains the following to equations of $e_3, e_6$, with coefficients in $e_1, e_2, e_7, K, K'$---all known quantities.

\begin{dmath*}
7257600e_1\widetilde\Lambda_8^7 = 3{e}_{3}^{3}+(3{e}_{1}^{3}-6{e}_{1}{e}_{2}){e}_{3}^{2}+(42{e}_{1}^{2}-15{e}_{2}-6){e}_{3}{e}_{6}+4{e}_{1}{e}_{6}^{2}-6{e}_{1}K'{e}_{3}+(-154{e}_{1}^{5}+49{e}_{1}^{3}{e}_{2}-6{e}_{1}{e}_{2}^{2}+6{e}_{1}{e}_{2}){e}_{6}+20{e}_{1}^{9}-70{e}_{1}^{7}{e}_{2}+100{e}_{1}^{5}{e}_{2}^{2}-50{e}_{1}^{3}{e}_{2}^{3}+4{e}_{1}{e}_{2}^{4}+34{e}_{1}^{4}K'-8{e}_{1}^{2}{e}_{2}K'-6{e}_{1}^{2}{e}_{7}
\end{dmath*}

\begin{dmath*}
1916006400e_1^2 \widetilde\Lambda_{10}^7 = 5 {e}_{3}^{4}+(112 {e}_{1}^{3}-46 {e}_{1} {e}_{2}) {e}_{3}^{3}+(120
      {e}_{1}^{2}-50 {e}_{2}-20) {e}_{3}^{2} {e}_{6}+(112 {e}_{1}^{6}-250
      {e}_{1}^{4} {e}_{2}+72 {e}_{1}^{2} {e}_{2}^{2}-20 {e}_{1} {K'})
      {e}_{3}^{2}+(1568 {e}_{1}^{5}-1164 {e}_{1}^{3} {e}_{2}-264 {e}_{1}^{3}+230
      {e}_{1} {e}_{2}^{2}+92 {e}_{1} {e}_{2}) {e}_{3} {e}_{6}+(820 {e}_{1}^{4}-620
      {e}_{1}^{2} {e}_{2}-280 {e}_{1}^{2}+125 {e}_{2}^{2}+100 {e}_{2}+20)
      {e}_{6}^{2}+(-224 {e}_{1}^{4} {K'}+92 {e}_{1}^{2} {e}_{2} {K'}+40
      {e}_{1}^{2} {e}_{7}) {e}_{3}+(-1600 {e}_{1}^{8}+1788 {e}_{1}^{6} {e}_{2}-414
      {e}_{1}^{4} {e}_{2}^{2}+224 {e}_{1}^{4} {e}_{2}+32 {e}_{1}^{2} {e}_{2}^{3}-52
      {e}_{1}^{2} {e}_{2}^{2}-232 {e}_{1}^{3} {K'}+100 {e}_{1} {e}_{2}
      {K'}+40 {e}_{1} {K'}) {e}_{6}+48 {e}_{1}^{12}-216 {e}_{1}^{10}
      {e}_{2}+448 {e}_{1}^{8} {e}_{2}^{2}-420 {e}_{1}^{6} {e}_{2}^{3}+144
      {e}_{1}^{4} {e}_{2}^{4}+304 {e}_{1}^{7} {K'}-8 {e}_{1}^{2}
      {e}_{2}^{5}-292 {e}_{1}^{5} {e}_{2} {K'}-224 {e}_{1}^{5} {e}_{7}+32
      {e}_{1}^{3} {e}_{2}^{2} {K'}+52 {e}_{1}^{3} {e}_{2} {e}_{7}+20
      {e}_{1}^{2} {K'}^{2}
\end{dmath*}

The symmetry of solutions of polynomials is harder to exploit here than the codimension 6 case.  Computer trials evaluating $e_1, e_2, e_7, K, K'$ to integers suggest that the degree of the ideal generated by the two resulting polynomials in $e_3,e_6$ is 5.

\bigskip
\subsection*{Acknowledgements} We would like to thank Justin Chen for suggesting the problem and providing help with Macaulay2 codes, Mengyuan Zhang for helpful discussions and code contributions, David Eisenbud for various helpful conversations, Burt Totaro for pointing to virtual Todd genera, and lastly an anonymous reference for helpful comments.

\end{document}